\documentclass[reqno]{amsart}

\usepackage{a4wide}
\usepackage{color}
\usepackage{mathrsfs}
\usepackage{mathtools}
\usepackage{tikz-cd}
\usepackage{amsmath}
\usepackage{amssymb}
\usepackage{bbm}
\numberwithin{equation}{section}
\usepackage[colorlinks,citecolor=green,linkcolor=red]{hyperref}
\usepackage{hyphenat}
\usepackage{bm}
\usepackage{comment}

\usepackage[latin1]{inputenc}

\newcommand{\N}{\mathbb{N}}
\newcommand{\R}{\mathbb{R}}

\renewcommand{\H}{{\rm H}}

\renewcommand{\d}{{\mathrm d}}
\newcommand{\e}{{\rm e}}
\newcommand{\restr}[1]{\lower3pt\hbox{\(|_{#1}\)}}

\newcommand{\nchi}{{\raise.3ex\hbox{\(\chi\)}}}
\newcommand{\Der}{{\rm Der}}
\newcommand{\1}{\mathbbm 1}
\newcommand{\fr}{\penalty-20\null\hfill\(\blacksquare\)}
\newcommand{\mm}{\mathfrak{m}}
\newcommand{\X}{{\rm X}}

\newcommand{\dist}{{\rm dist}}
\newcommand{\ppi}{\boldsymbol{\pi}}

\newcommand{\LIP}{{\rm LIP}}
\newcommand{\Lip}{{\rm Lip}}

\renewcommand{\div}{{\rm div}}

\newcommand{\md}{\operatorname{md}}

\newcommand{\ud}{\operatorname{d}}

\newcommand{\inv}{^{-1}}

\newtheorem{theorem}{Theorem}[section]
\newtheorem{corollary}[theorem]{Corollary}
\newtheorem{lemma}[theorem]{Lemma}
\newtheorem{proposition}[theorem]{Proposition}
\newtheorem{definition}[theorem]{Definition}

\newtheorem{remark}[theorem]{Remark}

\title[Spaces with Riemannian curvature bounds are UIH]{Spaces with Riemannian curvature bounds \\ are universally infinitesimally Hilbertian}

\author{Jes\'{u}s N\'{u}\~{n}ez-Zimbr\'{o}n}
\address{Facultad de Ciencias, UNAM, Mexico}
\email{nunez-zimbron@ciencias.unam.mx}

\author{Enrico Pasqualetto}
\address{Department of Mathematics and Statistics,
P.O.\ Box 35 (MaD), FI-40014 University of Jyvaskyla}
\email{enrico.e.pasqualetto@jyu.fi}

\author{Elefterios Soultanis}
\address{Department of Mathematics and Statistics,
P.O.\ Box 35 (MaD), FI-40014 University of Jyvaskyla}
\email{elefterios.e.soultanis@jyu.fi}

\begin{document}
\date{\today} 

\begin{abstract}
We show that a metric space $X$ that, at every point, has a Gromov--Hausdorff tangent with the splitting property (i.e.\ every geodesic line splits off a factor $\R$), is universally infinitesimally Hilbertian (i.e.\ $W^{1,2}(X,\mu)$ is a Hilbert space for every measure $\mu$). This connects the infinitesimal geometry of $X$ to its analytic properties and is, to our knowledge, the first general criterion guaranteeing universal infinitesimal Hilbertianity. Using it we establish universal infinitesimal Hilbertianity of finite dimensional RCD-spaces. We moreover show that (possibly infinite dimensional) Alexandrov spaces are universally infinitesimally Hilbertian and construct an isometric embedding of tangent modules.
\end{abstract}
\maketitle
\section{Introduction}
A metric measure space $(X,\mu)$ is said to be \emph{infinitesimally Hilbertian} if the Sobolev space $W^{1,2}(X,\mu)$
is a Hilbert space (see Section \ref{sec:sobolev} for the precise definitions). Infinitesimal Hilbertianity was introduced by Gigli in \cite{Gig:15} and is a crucial part of the definition of spaces with Riemannian curvature bounds (RCD-spaces), see e.g.\ \cite{Amb:Gig:Sav:14-2,Amb:Gig:Sav:14}. It has since become of independent interest. One reason for this is that it connects Sobolev spaces over metric spaces to the theory of Dirichlet forms, which has an extensive theory with a myriad of applications e.g.\ in analysis, probability and the theory of fractals. This connection was used, for example, in \cite{EB-raj-sou24} to make progress on the tensorization problem for Sobolev spaces. 

Infinitesimal Hilbertianity also arises in the analysis of \(2\)-Wasserstein spaces \cite{Fo:Sa:So:23} and of spaces of non-negative measures endowed with the Hellinger--Kantorovich distance \cite{DS:So:25}. For example, in topological data-analysis, the space of persistence diagrams equipped with the Wasserstein distance is an Alexandrov space \cite{Tur:Mil:Muk:Har:14} (and thus falls within the scope of Corollary \ref{cor:rcd-alex-inf-hilb} below). We refer the reader e.g.\ to \cite{Fo:He:So:2025,DS:So:25} for a detailed account of some of the many more (analytic, geometric and computational) motivations and applications of analysis on Wasserstein spaces. Another concept related to infinitesimal Hilbertianity, property (ET) (Euclidean tangents) \cite{Lyt:Wen:17}, arises in connection with the Plateau-problem in metric spaces; solutions to the Plateau-problem in spaces with property (ET) are weakly conformal.

In these and other contexts, there is (arguably) no natural choice of measure on the metric space for which infinitesimal Hilbertianity is clear. Thus the question arises: for which metric spaces $X$ is $(X,\mu)$ infinitesimally Hilbertian \emph{for any measure $\mu$}? Such metric spaces are called \emph{universally infinitesimally Hilbertian}.

\begin{definition}\label{def:uih}
A complete metric space $X$ is said to be \emph{universally infinitesimally Hilbertian (UIH)} if $W^{1,2}(X,\mu)$ is a Hilbert space for any Radon measure $\mu$ on $X$
that is \emph{boundedly finite} (i.e.\ finite on bounded sets).
\end{definition}
So far, all known examples arise from very precise knowledge of the metric geometry of the space: Euclidean spaces \cite{Gig:Pas:19}, Riemannian manifolds \cite{Lu:Pa:20}
and Wasserstein spaces over them \cite{Fo:Sa:So:23}, CAT($\kappa$)-spaces \cite{DiMar:Gig:Pas:Sou:20}, and sub-Riemannian manifolds \cite{LeDo:Lu:Pa:23}; see Remark \ref{rmk:examples_UIH}
for more details.
Notably, universal infinitesimal Hilbertianity is not known in the literature for nonsmooth spaces with lower curvature bounds. While for sectional (Alexandrov) curvature lower bounds UIH is expected, for RCD-spaces UIH is, at first glance, perhaps surprising since the RCD condition involves a specific reference measure and there is no a priori reason why $W^{1,2}(X,\mu)$ should be a Hilbert space for other measures $\mu$. Nevertheless, in this paper we establish UIH in both cases through a more general principle. Namely, we show that if blow-ups of Lipschitz curves split off an $\R$-factor, then the space is universally infinitesimally Hilbertian.

\begin{theorem}\label{thm:inf-split-implies-inf-hilb}
Suppose a complete metric space $X$ is weakly infinitesimally curve-splitting. Then $X$ is universally infinitesimally Hilbertian.
\end{theorem}

A metric space $X$ is said to be \emph{weakly infinitesimally curve-splitting} if, for every $x\in X$, there exists $(Y_x,o_x)\in\operatorname{Tan}(X,x)$ satisfying the splitting property for blow-ups of curves: if $\gamma$ is a Lipschitz curve then for a.e. $t$, the blow-up $\gamma_\infty:\R\to Y_{\gamma(t)}$ of $\gamma$ at $t$ is geodesic splitting line, i.e. there exists a metric space $Y'$ and an isometric homeomorphism $\iota:Y\to \R\times Y'$ such that $P_{Y'}(\iota(\gamma_\infty(\R)))$ is a singleton.
Here, $\operatorname{Tan}(X,x)$ denotes the collection of pointed Gromov--Hausdorff tangents and ultratangents to $X$ at $x$ (see Section \ref{sec:GH}), whereas $P_{Y'}\colon\R\times Y'\to Y'$ is the projection to $Y'$. See Section \ref{sec:inf-split}.

\begin{remark}\label{rmk:ET}
{\rm
A closely related condition, namely that any $2$-plane contained in an ultratangent of $X$ is Euclidean, is known to imply property (ET), see \cite[Definition 11.1 and Proposition 11.2]{Lyt:Wen:17}. In fact, (ET) turns out to be implied by universal infinitesimal Hilbertianity, and thus also by weak infinitesimal curve-splitting as formulated in this paper, see Theorem \ref{thm:UIH-implies-ET}.
\fr}\end{remark}

Theorem \ref{thm:inf-split-implies-inf-hilb} yields universal infinitesimal Hilbertianity for a large class of metric spaces including, but not limited to, spaces with Riemannian lower curvature bounds (RCD-spaces). However, since Gigli's splitting theorem fails for RCD($0,\infty$)-spaces, the infinite dimensional case is not covered by the theorem. The situation with Alexandrov spaces (that is, length metric spaces of sectional curvature bounded below in the sense of geodesic triangle comparison) is slightly more subtle: while the splitting theorem is known for Alexandrov spaces with curvature $\ge 0$, the tangent cone of an infinite dimensional Alexandrov space, while satisfying the triangle comparison condition, is not necessarily a length space (see the work of Halbeisen \cite{Halb00} for an example of this type). 

However, any ultratangent of an Alexandrov space is an Alexandrov space with curvature $\ge 0$, so that Theorem \ref{thm:inf-split-implies-inf-hilb} covers infinite dimensional Alexandrov spaces as well. This subtlety highlights the different uses of Gromov--Hausdorff and ultratangents.

\begin{corollary}\label{cor:rcd-alex-inf-hilb}
If \((X,d,\mathfrak m)\) is an ${\rm RCD}(K,N)$-space for some $N\in [1,\infty)$ and $K\in\R$, or if \((X,d)\) is an Alexandrov space, then $(X,d)$ is universally infinitesimally Hilbertian.
\end{corollary}

We do not know whether RCD($K,\infty$)-spaces are universally infinitesimally Hilbertian. However, in the case of Alexandrov spaces, we can say slightly more about the structure of Gigli's $L^2$-tangent module by constructing an isometric embedding of $L^2(TX)$ into a \emph{geometric} tangent bundle $\Gamma_\mu(HX)$ -- for any measure $\mu$ -- in the spirit of \cite{DiMar:Gig:Pas:Sou:20}. This different technique is interesting in its own right and also implies Corollary \ref{cor:rcd-alex-inf-hilb} for Alexandrov spaces. We refer to Section \ref{sec:alex} for the precise definitions.

\begin{theorem}\label{thm:alex-geom-bundle}
Let $(X,d)$ be an Alexandrov space and $\mu$ a boundedly-finite Radon measure on $X$.
Then there exists an \(L^\infty(\mu)\)-linear embedding map 
\[
\iota\colon L^2(TX)\hookrightarrow\Gamma_\mu(HX)
\]
satisfying \(|\iota(v)|=|v|\) for every \(v\in L^2(TX)\).
\end{theorem}

If \((X,d)\) is a compact Alexandrov space with non-negative sectional curvature, then the \(2\)-Wasserstein space \((\mathscr P(X),W_2)\) over \(X\)
is also a non-negatively curved Alexandrov space (and the converse implications holds as well, see \cite{Stu:06I}); consequently, it follows directly from Corollary \ref{cor:rcd-alex-inf-hilb}
that \((\mathscr P(X),W_2)\) is universally infinitesimally Hilbertian. If instead \((X,d)\) is a compact Alexandrov space with \({\rm Sect}\geq\kappa\) for some \(\kappa<0\),
but it does not hold that \({\rm Sect}\geq 0\), then \((\mathscr P(X),W_2)\) is \emph{not} an Alexandrov space (again, by \cite{Stu:06I}). However, even in this case one has
that \((\mathscr P(X),W_2)\) has `Euclidean tangent cones' in a suitable sense (see e.g.\ \cite{Oh:09,Gig:11,Gig:Oh:12}). For this reason we expect that \((\mathscr P(X),W_2)\)
is universally infinitesimally Hilbertian for \emph{any} compact Alexandrov space $X$, but in this paper we do not investigate further in this direction.
\subsection*{Acknowledgements}
We would like to thank Alexander Lytchak and Sylvester Eriksson-Bique for invaluable discussions regarding infinitesimal splitting. We are also grateful to Nicola Gigli for sharing an example of the failure of the splitting theorem for RCD($0,\infty$)-spaces, and to Nicola Cavallucci for useful comments on an earlier draft. 

J.N.-Z.\ was financially supported by PAPIIT-UNAM  IA103925. E.P.\ was financially supported by the Research Council of Finland grant no.\ 362898, and E.S. by the Research Council of Finland
grant no.\ 355122.  E.S.\ also gratefully acknowledges the hospitality of the special trimester on Metric Analysis funded by DFG under Germany's Excellence Strategy -EXC-2047/1- 390685813.

\section{Preliminaries}
\subsection{Sobolev spaces and derivations}\label{sec:sobolev}
In this paper, by a \emph{metric measure space} \((X,\mu)\) we mean a complete metric space \(X=(X,d)\) together with a
Radon measure \(\mu\) on \(X\) that is boundedly finite (i.e.\ finite on all bounded sets). Note that we are not requiring
separability of the metric space, but the measure is assumed to be Radon; this axiomatisation differs from that of most papers
on metric Sobolev spaces. The \emph{Sobolev space} $W^{1,2}(X,\mu)$ we consider is the one introduced in \cite{Amb:Gig:Sav:14}
as a variant of \cite{Ch:99}: it is defined via the $L^2$-relaxation of the functional.
\begin{align*}
L^2(\mu)\ni f\mapsto\inf\bigg\{\int_X(\Lip_a\bar f)^2\,\d\mu\;\bigg|\;\bar f\in\LIP_{bs}(X),\,f=\bar f\text{ }\mu\text{-a.e.}\bigg\}\in[0,+\infty].
\end{align*}
Here, \(\LIP_{bs}(\X)\) denotes the space of all real-valued Lipschitz functions on \(X\) having bounded support,
while \(\Lip_a f\colon X\to[0,+\infty)\) is the \emph{asymptotic slope} function of \(f\), which is defined as
\[
\Lip_a f(x)\coloneqq\inf\big\{\Lip(f|_{B_r(x)})\;\big|\;r>0\big\}\quad\text{ for every }x\in X.
\]
As proved in \cite{Amb:Gig:Sav:13}, on arbitrary metric measure spaces the above notion of Sobolev space is equivalent
to the \emph{Newtonian--Sobolev space} \(N^{1,2}(X)\) \cite{Sha:00} (see also \cite{HKST:15,Bj:Bj:11}) and to the
Sobolev space defined in terms of \emph{test plans}. For the purposes of this paper, however, we mostly use another
equivalent approach via integration by parts introduced in \cite{DiMar:14,DiMaPhD:14} (see also \cite{Amb:Iko:Luc:Pas:24}).
\begin{remark}\label{rmk:non-sep_mms}{\rm
In fact, the equivalence of the various definitions of Sobolev spaces was proved for complete separable metric spaces
equipped with a boundedly-finite Borel (thus, Radon) measure. However, the equivalence in the larger class of metric
measure spaces under consideration in this paper can be deduced as follows. Given a metric measure space \((X,\mu)\) in
our sense, we have that the support \({\rm spt}(\mu)\) is a closed (thus, complete) and separable subset of \(X\), and that
\(\mu\) is concentrated on \({\rm spt}(\mu)\); here, we are using the assumption that \(\mu\) is a Radon measure. Moreover,
the Sobolev spaces \(W^{1,2}(X,\mu)\) and \(W^{1,2}({\rm spt}(\mu),\mu)\) (with \(W^{1,2}\) defined according to any of the
approaches we mentioned above) can be canonically identified: for the notion via relaxation this is shown in
\cite{DiMa:Gig:Pra:20}; for the notion via test plans it is an easy consequence of the definition,
whence it follows also for the other two notions thanks to the equivalence results of \cite{Amb:Gig:Sav:13}.
See \cite[Section 7]{DiMar:Gig:Pas:Sou:20} for a related discussion on metric Sobolev spaces over non-separable
metric measure spaces.
\fr}\end{remark}

According to \cite{DiMaPhD:14}, a \emph{derivation} on $(X,\mu)$ is a linear operator $b:\LIP_{bs}(X)\to L^0(\mu)$ such that
(i) $b(fg)=fb(g)+gb(f)$ for all $f,g\in \LIP_{bs}(X)$, and (ii) there exists $g\in L^2(\mu)^+$ so that $|b(f)|\le g\,\Lip_a f$
in the \(\mu\)-a.e.\ sense for all $f\in \LIP_{bs}(X)$. The \(\mu\)-a.e.\ smallest such $g\in L^2(\mu)^+$ is denoted by $|b|$.
We denote by \(\Der^2(\X,\mu)\) the space of all derivations on \((X,\mu)\), which is a Banach space with respect to the norm
\(\|b\|_{\Der^2(X,\mu)}\coloneqq\||b|\|_{L^2(\mu)}\). Moreover, we say that a given derivation \(b\in\Der^2(X,\mu)\)
has \emph{divergence} if there exists a function \(\div(b)\in L^2(\mu)\) such that
\[
\int b(f)\,\d\mu=-\int f\,\div(b)\,\d\mu\quad\text{ for every }f\in\LIP_{bs}(\X).
\]
We denote by \(\Der^2_2(X,\mu)\) the space of all derivations on \((X,\mu)\) having divergence. Observe that,
if $h\in \LIP_{bs}(X)$ and $b\in \Der^2_2(X,\mu)$, then $hb\in \Der^2_2(X,\mu)$ and \(\div(hb)=h\,\div(b)+b(h)\).
\begin{remark}{\rm
Derivations with divergence enjoy nicer properties than arbitrary derivations. For example, they are `strongly local',
meaning that the identity \(b(f)=b(g)\) holds \(\mu\)-a.e.\ on \(\{f=g\}\) for every \(b\in\Der^2_2(X,\mu)\) and
\(f,g\in\LIP_{bs}(X)\). Hence, given any function \(h\in\LIP(X)\), there exists a unique \(b(h)\in L^2(\mu)\) such that
\[
\1_{\{\eta=1\}}b(h)=\1_{\{\eta=1\}}b(\eta h)\quad\text{ for every }\eta\in\LIP_{bs}(X).
\]
In particular, it makes sense to consider \(b(\dist_y)\) for any \(y\in X\), where \(\dist_y\coloneqq d(\cdot,y)\).
\fr}\end{remark}
\begin{definition}[Sobolev space via integration by parts]
Let $f\in L^2(\mu)$ be given. Then we declare that $f\in W^{1,2}(X)$ if and only if there exists a linear operator
\(L_f\colon\Der^2_2(X,\mu)\to L^1(\mu)\) such that the following properties are satisfied:
\begin{itemize}
\item[\(\rm i)\)] There exists a function \(H\in L^2(\mu)^+\) such that
\[
|L_f(b)|\leq H|b|\quad\text{ for every }b\in\Der^2_2(X,\mu).
\]
\item[\(\rm ii)\)] It holds that \(L_f(hb)=h\,L_f(b)\) for every \(h\in\LIP_{bs}(X)\) and \(b\in\Der^2_2(X,\mu)\).
\item[\(\rm iii)\)] \textsc{Integration-by-parts formula.} It holds that
\[
\int L_f(b)\,\d\mu=-\int f\,\div(b)\,\d\mu\quad\text{ for every }b\in\Der^2_2(X,\mu).
\]
\end{itemize}
\end{definition}
The minimal $L^2$-function $H$ satisfying i) above is called the \emph{minimal relaxed slope}
(or the \emph{minimal weak upper gradient}) of \(f\), and denoted by $|Df|$. It can be readily checked that
\(L_f\) is uniquely determined, and that
\begin{equation}\label{eq:formula_|Df|}
|Df|=\bigvee_{b\in\Der^2_2(X,\mu)}\1_{\{|b|>0\}}\frac{L_f(b)}{|b|},
\end{equation}
where $\bigvee$ denotes the measure-theoretic essential supremum, cf.\ \cite{Fre:12}. The norm in $W^{1,2}(X,\mu)$ is
\[
\|f\|_{1,2}=\big(\|f\|^2_{L^2(\mu)}+\||Df|\|_{L^2(\mu)}^2\big)^{1/2}\quad\text{ for every }f\in W^{1,2}(X,\mu),
\]
and $(W^{1,2}(X,\mu),\|\cdot\|_{1,2})$ is a Banach space. However, \(W^{1,2}(X,\mu)\) is not always a Hilbert space;
when it does, the metric measure space \((X,\mu)\) is said to be \emph{infinitesimally Hilbertian}, according to
\cite[Definition 4.19]{Gig:15}. Whereas infinitesimal Hilbertianity is a property of metric measure spaces, universal infinitesimal Hilbertianity (see Definition \ref{def:uih}) is a property of a metric space. To our knowledge, this terminology was introduced in \cite{DiMar:Gig:Pas:Sou:20}.

\begin{remark}[Known examples of UIH spaces]\label{rmk:examples_UIH}{\rm
We provide a list of metric spaces that are known to be universally infinitesimally Hilbertian:
\begin{itemize}
\item[\(\rm i)\)] Euclidean spaces \cite{Gig:Pas:21} (see also \cite{DiMa:Lu:Pa:20} for an alternative proof).
\item[\(\rm ii)\)] Riemannian manifolds and Carnot groups \cite{Lu:Pa:20}.
\item[\(\rm iii)\)] Hilbert spaces \cite{Sav:22}.
\item[\(\rm iv)\)] Locally \({\rm CAT}(\kappa)\) spaces \cite{DiMar:Gig:Pas:Sou:20}.
\item[\(\rm v)\)] Sub-Riemannian manifolds \cite{LeDo:Lu:Pa:23}.
\item[\(\rm vi)\)] The \(2\)-Wasserstein space over a Riemannian manifold or a Hilbert space \cite{Fo:Sa:So:23}.
\item[\(\rm vii)\)] The space of finite non-negative Borel measures over a Riemannian manifold equipped with the
Hellinger--Kantorovich distance \cite{DS:So:25}.
\end{itemize}
As discussed in the introduction, the main goal of this paper is to show that also
Alexandrov spaces and finite-dimensional \({\rm RCD}\) spaces are universally infinitesimally Hilbertian.
\fr}\end{remark}

We point out that
\begin{equation}\label{eq:Hilb_UIH_Banach}
\emph{a Banach space is universally infinitesimally Hilbertian if and only if it is Hilbert.}
\end{equation}
For the `if' implication in \eqref{eq:Hilb_UIH_Banach}, recall iii) of Remark \ref{rmk:examples_UIH}. 
For the `only if', one can argue as follows: if \((X,\|\cdot\|)\) is a non-Hilbertian
Banach space, then some \(2\)-dimensional vector subspace \(V\) of \(X\) is not Hilbert. Hence, letting \(\mu\) be the
restriction of the \(2\)-dimensional Hausdorff measure to \(V\), it can be readily checked that
\((X,\|\cdot\|,\mu)\) is not infinitesimally Hilbertian, so that \(X\) is not UIH.
\subsection{Normed modules}\label{sec:normed_mod}
Let \((X,\mu)\) be a metric measure space. Let \(\mathscr M\) be an \emph{\(L^2(\mm)\)-normed \(L^2(\mm)\)-module} in the sense of
\cite[Definition 1.2.10]{Gig:18} (see also \cite[Definition 2.3]{Gig:17}). Recall that \(\mathscr M\) is equipped with a multiplication
operation \(L^\infty(\mu)\times\mathscr M\ni(f,v)\mapsto f\cdot v\in\mathscr M\) and a pointwise norm \(|\cdot|\colon\mathscr M\to L^2(\mu)^+\).
In particular, \(\mathscr M\) is a Banach space with respect to \(\|v\|_{\mathscr M}\coloneqq\||v|\|_{L^2(\mu)}\).
\medskip

When \(\mathscr M\) is a Hilbert space (which, by \cite[Proposition 1.2.21]{Gig:18}, is equivalent to requiring that \(|v+w|^2+|v-w|^2=2|v|^2+2|w|^2\)
for every \(v,w\in\mathscr M\)), we denote by \(\mathscr M\times\mathscr M\ni(v,w)\mapsto\langle v,w\rangle\in\R\) the scalar product of \(\mathscr M\)
(not the pointwise scalar product defined in \cite{Gig:18} right after Remark 1.2.23). It follows from \cite[Proposition 1.2.13 and Theorem 1.2.24]{Gig:18}
that for any \(w\in\mathscr M\) there exists a unique \(L^\infty(\mm)\)-linear and continuous operator \(L_w\colon\mathscr M\to L^1(\mu)\) such that
\begin{equation}\label{eq:Int}
\langle w,v\rangle=\int L_w(v)\,\d\mu\quad\text{ for every }v\in\mathscr M.
\end{equation}
Moreover, there exists a function \(g\in L^2(\mu)^+\) such that \(|L_w(v)|\leq g|v|\) holds for every \(v\in\mathscr M\), and the \(\mu\)-a.e.\ minimal
such function \(g\), which we denote by \(|L_w|\), satisfies \(\||L_w|\|_{L^2(\mu)}=\|w\|_{\mathscr M}\).
\medskip

Following \cite[Section 4.2]{Amb:Iko:Luc:Pas:24}, we give the ensuing definition:
\begin{definition}[Lipschitz tangent module]
Let \((X,\mu)\) be a metric measure space. Then we define its \emph{Lipschitz tangent module} \(L^2_\Lip(TX)\) as the closure of \(\Der^2_2(X,\mu)\) in \(\Der^2(X,\mu)\).
\end{definition}
One can readily check (cf.\ \cite{Amb:Iko:Luc:Pas:24}) that \(\Der^2(X,\mu)\) and \(L^2_\Lip(TX)\) are \(L^2(\mu)\)-normed \(L^2(\mu)\)-modules with respect
to the natural pointwise operations. As proved in \cite[Proposition 6.5]{DiMar:Gig:Pas:Sou:20}, we have that
\begin{equation}\label{eq:inf_Hilb_gives_Hilb_Liptgmod}
(X,\mu)\emph{ is infinitesimal Hilbertian}\quad\Longrightarrow\quad L^2_\Lip(TX)\cong L^2(TX),
\end{equation}
where \(L^2(TX)\) denotes the \emph{(Sobolev) tangent module} in the sense of \cite[Definition 2.3.1]{Gig:18}.
\begin{remark}\label{rmk:bdd_spt_der_dense}{\rm
Given any metric measure space \((X,\mu)\), we claim that
\[
\big\{b\in\Der^2_2(X,\mu)\;\big|\;|b|\text{ has bounded support}\big\}\text{ is dense in }L^2_\Lip(TX).
\]
To prove it, fix \(b\in\Der^2_2(X,\mu)\) and \(\bar x\in{\rm spt}(\mu)\). Letting
\(\eta_n\coloneqq\max\{1-d(B_n(\bar x),\cdot),0\}\in\LIP_{bs}(X)\) for every \(n\in\N\), we have that
\(\eta_n b\in\Der^2_2(X,\mu)\) and that the support of \(|\eta_n b|\) is bounded. Moreover,
\(\|b-\eta_n b\|_{\Der^2(X,\mu)}\to 0\) as \(n\to\infty\) by the dominated convergence theorem.
Since \(\Der^2(X,\mu)\) is dense in \(L^2_\Lip(TX)\) by the very definition of \(L^2_\Lip(TX)\), the claim follows.
\fr}\end{remark}
\subsection{Alexandrov spaces}\label{sec:alex}
We refer the interested reader to \cite{Alex:Kap:Pet:24,Bu:Bu:Iv:01} for extensive surveys on metric geometry and Alexandrov curvature bounds. Throughout this subsection and the rest of the paper, by \emph{Alexandrov space} we mean a complete, length metric space $(Y,d_Y)$ with CBB (curvature bounded below) in the triangle comparison sense, see \cite[Definition 8.2]{Alex:Kap:Pet:24}. We recall that the function
\[
\dist_y\text{ is \(1\)-Lipschitz on }Y\text{ and semiconcave on }Y\setminus\{y\}\text{ for every }y\in Y,
\]
e.g.\ as a consequence of \cite[Theorem 8.23]{Alex:Kap:Pet:24}. Here, we set \(\dist_y\coloneqq d_Y(\cdot,y)\).
\subsubsection*{Tangent cone and its linear subcone}
Given $p\in Y$, let $T_pY:=C(\Sigma_pY)$ be the tangent cone to $Y$ at $p$, given by the Euclidean cone over the space of directions $\Sigma_pY$. The space $\Sigma_pY$ is defined as the closure of (equivalence classes of) unit speed geodesics $\alpha:[0,\varepsilon)\to Y$ emanating from $p$ under the angle metric
\[
\angle(\alpha, \beta)=\angle_p(\alpha, \beta)=\lim_{t,s\to 0}\arccos\frac{d^2(p,\alpha_t)+d^2(p,\beta_s)-d^2(\alpha_t,\beta_s)}{2d(p,\alpha_t)d(p,\beta_s)}\in [0,\pi],
\]
which we extend to $T_pY$ by setting $\angle(tv,sw)=\angle(v,w)$ for $v,w\in \Sigma_pY$. The metric $d_p$ in $T_pY$ is given by 
\[
d_p^2(tv,sw)=t^2+s^2-2st\cos\angle(v,w).
\]

\begin{remark}{\rm
We do not assume here $Y$ to be finite dimensional, or even locally compact. Thus, $T_pY$ is not necessarily a length space, cf.\ \cite[Example 13.6]{Alex:Kap:Pet:24}. It is however contained isometrically in the ultratangent $T^\omega_pY$, which is a geodesic, non-negatively curved Alexandrov space \cite[Theorem 6.17, Theorem 13.1 and Observation 4.9]{Alex:Kap:Pet:24}.
\fr}\end{remark}

For $v\in T_pY$ set
\begin{align*}
|v|=|v|_p&:=d_p(o,v)\\
\langle v,w\rangle &:=|v||w|\cos\angle(v,w).
\end{align*}
Observe that, with this notation, we have $\langle v,v\rangle=|v|^2$. Given $v,w\in T_pY$, there exists a unique vector $as(v,w)\in T_pY$ -- called the \emph{anti-sum} of $v$ and $w$ -- with the following properties:
\begin{align}
0&\le \langle v,z\rangle+\langle w,z\rangle+\langle as(v,w),z\rangle, \quad z\in T_pY\label{eq:antisum1}\\
0&=\langle v,as(v,w)\rangle+\langle w,as(v,w)\rangle+|as(v,w)|^2\label{eq:antisum2};
\end{align}
see \cite[13.32]{Alex:Kap:Pet:24}. Applying this to $v,0\in T_pY$ we obtain the existence of a unique \emph{polar} $v^*\in T_pY$ of $v$ characterized by (a) $|v^*|^2+\langle v,v^*\rangle=0$ and (b) $0\le \langle v,z\rangle+\langle v^*,z\rangle$ for all $z\in T_pY$.

\begin{definition}
A vector $v\in T_pY$ has an \emph{opposite vector} $w\in T_pY$ if $|v|=|w|$ and $\angle(u,v)=\pi$. We denote by $H_pY\subset T_pY$ the collection of all vectors which have an opposite vector, and equip it with the metric inherited from $T_pY$.
\end{definition}
Opposite vectors, if they exist, are unique \cite[Proposition 13.38]{Alex:Kap:Pet:24} (hence we denote the opposite of $v\in H_pY$ by $-v$) and are characterized by
\begin{align*}
\langle v,z\rangle+\langle -v,z\rangle=0,\quad z\in T_pY;
\end{align*}
see \cite[Proposition 13.37]{Alex:Kap:Pet:24}. From this it is easy to see that $-(t\cdot v)=t\cdot (-v)$\footnote{In fact more generally $t\cdot as(v,w)=as(t\cdot v,t\cdot w)$ for $v,w\in T_pY$ and $t\ge 0$.} and that $d(-v,-w)=d(v,w)$. In particular $H_pY$ is a subcone of $T_pY$ and is moreover closed.

\begin{remark}\label{rmk:polar-opposite}{\rm
Note that every $v\in T_pY$ has a polar but not necessarily an opposite. If the opposite $-v$ of $v$ exists, then it is the polar. In fact $v^*=-v$ (i.e.\ the polar is the opposite vector) if and only if $|v|=|v^*|$. This is easily deduced from the identity \(|v|^2=-\langle v,v^*\rangle=-|v||v^*|\cos\angle(v,v^*)\).
\fr}\end{remark}

\subsubsection*{Differential and gradient}
Let $\Omega\subset Y$ be open and $f:\Omega\to\R$ be a locally Lipschitz semiconcave function \cite[Definition 3.17]{Alex:Kap:Pet:24}. Then, for all $p\in\Omega$, $f$ admits a differential $\ud_pf:T_pY\to\R$, which is determined by
\begin{itemize}
    \item[(i)] $\ud_pf(\alpha^+)=\lim_{h\to 0^+}\frac{f(\alpha_h)-f(p)}{h}$, $\alpha^+ \in \Sigma_pY$. Observe that in turn, $\ud_pf(v)$ can be computed for every $v\in \Sigma_pY$ by (approximation with) geodesic directions.
    \item[(ii)] $\ud_pf(tv)=t\ud_pf(v)$, $v\in \Sigma_pY$, $t\ge 0$. 
\end{itemize}

 Moreover, for every $p\in\Omega$, there exists a unique vector $\nabla_pf\in T_pY$ such that
\begin{align*}
\ud_pf(\nabla_pf)=|\nabla_pf|^2\quad\textrm{and}\quad \ud_pf(v)\le \langle\nabla_pf,v\rangle,\quad v\in T_pY;
\end{align*}
see \cite[13.17 and 13.20]{Alex:Kap:Pet:24}.

\begin{remark}\label{rmk:restr-of-concave-fn}{\rm
Since $T_pY$ might not be a length space, we cannot really say that $\ud_pf$ is concave. However, $\ud_pf$ is the restriction to $T_pY$ of the ultradifferential $\ud_p^\omega f:T^\omega_pY\to \R$, which is a concave function; cf.\ \cite[6.6 and 13.1]{Alex:Kap:Pet:24}.
\fr}\end{remark}
\subsection{Tangents of metric spaces}\label{sec:GH} We refer the reader to \cite[Chapter 7 and 8]{Bu:Bu:Iv:01} and \cite[Chapter 4]{Alex:Kap:Pet:24} for the definition and basic properties of Gromov--Hausdorff (GH) limits and ultralimits, respectively. Recall that, given a metric space $X$, a GH (resp. ultra) tangent of $X$ at a point $x\in X$ is a GH (resp. ultra)limit of the sequence $(X,r_j\inv d,x)$ for some sequence of scales $r_j\to 0$. We will denote the rescaled metric space $(X,r_j\inv d)$ by $r_j\inv X$. 

\begin{remark}{\rm
In general, the collection $\operatorname{Tan}_{GH}(X,x)$ of GH tangents of $X$ at $x\in X$ (up to pointed isometric homeomorphism) may be empty, while the collection $\operatorname{Tan}_\omega(X,x)$ of ultratangents, or $\omega$-tangents (with respect to a non-principal ultrafilter $\omega$) is always non-empty. The collections coincide for all $x\in X$ e.g. when $X$ is a doubling metric space. If $X$ is a finite dimensional Alexandrov space, then $\operatorname{Tan}_{GH}(X,p)=\{(T_pX,p)\}$ for every $p\in X$, where $T_pX$ is the tangent cone at $p$.
\fr}\end{remark}
Similarly, a blow-up $\psi:(Y,o)\to \R^k$ of a Lipschitz function $\varphi:X\to \R^k$ at a point $x\in X$ is the GH (resp. ultra) limit of the sequence $\frac{\varphi-\varphi(x)}{r_j}:(r_j\inv X,0)\to \R^k$ for some $(r_j)$ converging to zero. The blow-up $\psi$ is always a Lipschitz function with $\LIP(\psi)\le \Lip_a\varphi(x)$.

Given a Lipschitz curve $\gamma:[a,b]\to X$, $\gamma$ is metrically differentiable at $t$ for a.e. $t\in (a,b)$ (see \cite{Kir:94}). If $(Y,o)$ is a GH (resp. ultra) tangent with associated sequence of scales $(r_j)$, it follows that the maps 
\[
\gamma_j:r_j\inv([a,b]-t)\to (r_j\inv X,\gamma_t),\quad \gamma_j(s)=\gamma(t+r_js)
\]
have a GH (resp. ultra) limit $\gamma_\infty:\R\to (Y,o)$ satisfying
\[
d_Y(\gamma_\infty(s),\gamma_\infty(r))=|\gamma_t'||s-r|,\quad s,r\in \R,
\]
where $|\gamma_t'|:=\lim_{h\to\infty}\frac{d(\gamma_{t+h},\gamma_t)}{|h|}$ is the metric speed of $\gamma$ at $t$. Consequently $\gamma_\infty$ is a geodesic line with speed $|\gamma_t'|$. This fact is relevant in the definition of infinitesimal curve-splitting and will be useful in Section \ref{sec:inf-split}.

\section{Infinitesimal Hilbertianity via derivations}
The first result of this section, Theorem \ref{thm:suff_condit_IH}, shows that if \((\Der^2_2(X,\mu),\|\cdot\|_{\Der^2(X,\mu)})\)
is a pre-Hilbert space, then \(W^{1,2}(X,\mu)\) is a Hilbert space. For the reader's convenience we report its full proof,
which is essentially taken from \cite[Theorem 5.10]{Pas:Tai:25}.
\medskip

Let us fix some notation. Given two Banach spaces \(\mathbb V\) and \(\mathbb W\), we denote
by \(\mathbb V\times_2\mathbb W\) the product Banach space obtained by equipping \(\mathbb V\times\mathbb W\)
with the \(2\)-norm \(\mathbb V\times\mathbb W\ni(v,w)\mapsto(\|v\|_{\mathbb V}^2+\|w\|_{\mathbb W}^2)^{1/2}\).
Observe that the space \(\mathbb V\times_2\mathbb W\) is Hilbert if (and only if) both \(\mathbb V\) and \(\mathbb W\) are Hilbert,
and in this case we have that \(\langle(v,w),(\tilde v,\tilde w)\rangle=\langle v,\tilde v\rangle+\langle w,\tilde w\rangle\) for
every \(v,\tilde v\in\mathbb V\) and \(w,\tilde w\in\mathbb W\).
\begin{theorem}\label{thm:suff_condit_IH}
Let \((X,\mu)\) be a metric measure space satisfying
\begin{equation}\label{eq:parall_|b|}
|b+\tilde b|^2+|b-\tilde b|^2=2|b|^2+2|\tilde b|^2\quad\text{ for every }b,\tilde b\in\Der^2_2(X,\mu).
\end{equation}
Then \(L^2_\Lip(T X)\) is Hilbert. Letting \(\mathbb V\) be the linear subspace of
\(L^2(\mu)\times_2 L^2_\Lip(TX)\) given by
\[
\mathbb V\coloneqq\big\{(g,b)\in L^2(\mu)\times\Der^2_2(X,\mu)\;\big|\;g=\div(b)\big\},
\]
the Sobolev space \(W^{1,2}(X,\mu)\) is isometrically isomorphic to the orthogonal complement
\(\mathbb V^\perp\) of \(\mathbb V\) in the Hilbert space \(L^2(\mu)\times_2 L^2_\Lip(TX)\).
\end{theorem}
\begin{proof}
Integrating \eqref{eq:parall_|b|} we get that \((\Der^2_2(X,\mu),\|\cdot\|_{\Der^2(X,\mu)})\) is a pre-Hilbert space,
thus \(L^2_\Lip(TX)\) is a Hilbert space. For any \(f\in W^{1,2}(X,\mu)\), we have that \(\Der^2_2(X,\mu)\ni b\mapsto\int L_f(b)\,\d\mu\in\R\)
is a bounded linear operator, thus the Riesz representation theorem ensures that there exists a unique element \(w_f\in L^2_\Lip(TX)\) such that
\(\langle w_f,b\rangle=\int L_f(b)\,\d\mu\) for every \(b\in\Der^2_2(X,\mu)\). Using the results of Section \ref{sec:normed_mod},
it is easy to show that \(\|w_f\|_{L^2_\Lip(TX)}=\||Df|\|_{L^2(\mu)}\). Therefore, the linear map
\[
\phi\colon W^{1,2}(X,\mu)\to L^2(\mu)\times_2 L^2_\Lip(TX),\quad\phi(f)\coloneqq(f,w_f)
\]
is an isometry. We claim that \(\phi(W^{1,2}(X,\mu))=\mathbb V^\perp\). If \(f\in W^{1,2}(X,\mu)\) and \(b\in\Der^2_2(X,\mu)\), then
\[
\langle\phi(f),(\div(b),b)\rangle=\langle f,\div(b)\rangle+\langle w_f,b\rangle=\int f\,\div(b)\,\d\mu+\int L_f(b)\,\d\mu=0,
\]
which yields \(\phi(W^{1,2}(X,\mu))\subset\mathbb V^\perp\). To prove the converse inclusion, fix any element \((f,w)\in\mathbb V^\perp\).
Letting \(L_w\colon L^2_\Lip(TX)\to L^1(\mu)\) be as in \eqref{eq:Int}, we have that \(L_w|_{\Der^2_2(X,\mu)}\colon\Der^2_2(X,\mu)\to L^1(\mu)\)
satisfies \(|L_w(b)|\leq|L_w||b|\) for every \(b\in\Der^2_2(X,\mu)\), and \(L_w(hb)=h\,L_w(b)\) for all \(h\in\LIP_{bs}(X)\) (by the
\(L^\infty(\mu)\)-linearity of \(L_w\)). Finally, since \((\div(b),b)\in\mathbb V\) for all \(b\in\Der^2_2(X,\mu)\), it holds that
\[
\int f\,\div(b)\,\d\mu+\int L_w(b)\,\d\mu=\langle f,\div(b)\rangle+\langle w,b\rangle=\langle(f,w),(\div(b),b)\rangle=0,
\]
so that \(\int L_w(b)\,\d\mu=-\int f\,\div(b)\,\d\mu\). All in all, we proved that \(f\in W^{1,2}(X,\mu)\) and \(L_f=L_w\),
thus \((f,w)=\phi(f)\in\phi(W^{1,2}(X,\mu))\), whence the claim. As \(\phi\) is an isometry, the statement follows.
\end{proof}
\begin{corollary}\label{cor:equiv_cond_IH}
A metric measure space \((X,d,\mu)\) is infinitesimally Hilbertian if and only if its Lipschitz tangent module \(L^2_\Lip(TX)\) is Hilbert.
\end{corollary}
\begin{proof}
On the one hand, if \((X,d,\mu)\) is infinitesimally Hilbertian, then \(L^2(TX)\) is Hilbert by \cite[Proposition 2.3.17]{Gig:18}, and thus
\(L^2_\Lip(TX)\) is Hilbert by \eqref{eq:inf_Hilb_gives_Hilb_Liptgmod}. On the other hand, if \(L^2_\Lip(TX)\) is Hilbert, then \(W^{1,2}(X,\mu)\)
is Hilbert by Theorem \ref{thm:suff_condit_IH}.
\end{proof}

We conclude the section with a technical result, Lemma \ref{lem:formula_|b|^2+|tilde_b|^2}, which will be useful to prove that Alexandrov spaces
are UIH. First, let us give an auxiliary definition: given any metric measure space \((X,\mu)\), we say that a set \(S\subset\{f\in\LIP(X):\Lip(f)\leq 1\}\)
is \emph{norming for \(\Der^2_2(X,\mu)\)} if
\begin{equation}\label{eq:formula_|b|}
|b|=\bigvee_{f\in S}b(f)\quad\text{ for every }b\in\Der^2_2(X,\mu).
\end{equation}
It follows from \cite[Proposition 5.5]{DiMar:Gig:Pas:Sou:20} that if \(D\) is a dense subset of \({\rm spt}(\mu)\), then
\begin{equation}\label{eq:dist_y_norming}
\{\dist_y:y\in D\}\text{ is norming for }\Der^2_2(X,\mu).
\end{equation}
\begin{lemma}\label{lem:formula_|b|^2+|tilde_b|^2}
Let \((X,\mu)\) be a metric measure space. Let \(S\) be norming for \(\Der^2_2(X,\mu)\). Then
\begin{equation}\label{eq:formula_|b|^2+|tilde_b|^2}
\sqrt{|b|^2+|\tilde b|^2}=\bigvee_{\substack{\alpha,\beta\in[0,1]: \\ \alpha^2+\beta^2\leq 1}}\bigvee_{f,g\in S}b(\alpha f)+\tilde b(\beta g)
\quad\text{ for every }b,\tilde b\in\Der^2_2(X,\mu).
\end{equation}
\end{lemma}
\begin{proof}
Denote by \(\theta\in L^2(\mu)\) the right-hand side of \eqref{eq:formula_|b|^2+|tilde_b|^2}. Fix any
\(\varepsilon\in\mathbb Q\cap(0,1)\). Using \eqref{eq:formula_|b|}, we find a Borel \(\mu\)-a.e.\ partition
\((E_n)_{n\in\N}\) of \(\{|b|^2+|\tilde b|^2>0\}\), functions \((f_n)_{n\in\N},(g_n)_{n\in\N}\subset S\)
and \((\alpha_n)_{n\in\N},(\beta_n)_{n\in\N}\subseteq[0,1]\) with \(\alpha_n^2+\beta_n^2\leq 1\) such
that the following inequalities hold \(\mu\)-a.e.\ on \(E_n\):
\[\begin{split}
b(f_n)\geq(1-\varepsilon)|b|,&\qquad\tilde b(g_n)\geq(1-\varepsilon)|\tilde b|,\\
\bigg|\frac{|b|}{\sqrt{|b|^2+|\tilde b|^2}}-\alpha_n\bigg|\leq\varepsilon,
&\qquad\bigg|\frac{|\tilde b|}{\sqrt{|b|^2+|\tilde b|^2}}-\beta_n\bigg|\leq\varepsilon.
\end{split}\]
Therefore, we have that the following inequalities hold \(\mu\)-a.e.\ on \(E_n\):
\[\begin{split}
\sqrt{|b|^2+|\tilde b|^2}&=\frac{|b|^2+|\tilde b|^2}{\sqrt{|b|^2+|\tilde b|^2}}
\leq\frac{1}{1-\varepsilon}\frac{|b|}{\sqrt{|b|^2+|\tilde b|^2}}b(f_n)
+\frac{1}{1-\varepsilon}\frac{|\tilde b|}{\sqrt{|b|^2+|\tilde b|^2}}\tilde b(g_n)\\
&\leq\frac{\alpha_n\,b(f_n)+\beta_n\,\tilde b(g_n)}{1-\varepsilon}
+\varepsilon\frac{b(f_n)+\tilde b(g_n)}{1-\varepsilon}\leq
\frac{\theta}{1-\varepsilon}+\varepsilon\frac{|b|+|\tilde b|}{1-\varepsilon}.
\end{split}\]
Thanks to the arbitrariness of \(n\) and \(\varepsilon\), we deduce that \(\sqrt{|b|^2+|\tilde b|^2}\leq\theta\).
Conversely, fix any \(\alpha,\beta\in[0,1]\) with \(\alpha^2+\beta^2\leq 1\) and \(f,g\in S\). Then we have that
\[
b(\alpha f)+\tilde b(\beta g)\leq\alpha|b|+\beta|\tilde b|\leq
\sqrt{\alpha^2+\beta^2}\sqrt{|b|^2+|\tilde b|^2}\leq\sqrt{|b|^2+|\tilde b|^2},
\]
which gives \(\theta\leq\sqrt{|b|^2+|\tilde b|^2}\). All in all, \eqref{eq:formula_|b|^2+|tilde_b|^2} is proved.
\end{proof}
\section{Alexandrov spaces and Theorem \ref{thm:alex-geom-bundle}}

The goal of this section is to prove Theorem \ref{thm:alex-geom-bundle}.
\subsection{Geometric bundle of an Alexandrov space}
Recall the definition of the linear subcone from Section \ref{sec:alex}.
\begin{proposition}\label{prop:lin-subcone}
Let $Y$ be an Alexandrov space and $p\in Y$. Then $(H_pY,|\cdot|)$ is isometric to a Hilbert space. More precisely, for each $v,w\in H_pY$ we have that $v+w:=as(-v,-w)\in H_pY$; the map $(v,w)\mapsto v+w$ defines a vector addition which makes $H_pY$ into a vector space and $|\cdot|$ restricted to $H_pY$ is an inner product norm. 
\end{proposition}
This is an elaboration of the statement of \cite[Theorem 13.39]{Alex:Kap:Pet:24}. For simplicity we omit the dependence on $p$ of $d$ in the proof.
\begin{proof}
The proof of \cite[Theorem 13.39]{Alex:Kap:Pet:24} shows that $\zeta:=\frac{v+w}{2}=as(-v/2,-w/2)$ is a midpoint of $[v,w]$, i.e.\ $d(v,\zeta)=d(w,\zeta)=\frac 12 d(v,w)$. Since in \eqref{eq:antisum1} equality must hold for $z\in H_pY$ we obtain
\begin{align*}
\langle \zeta,v\rangle=\frac 12(|v|^2+\langle v,w\rangle),\quad \langle \zeta,w\rangle=\frac 12 (|w|^2+\langle v,w\rangle),\quad |\zeta|^2=\frac 12 (\langle v,\zeta\rangle+\langle w,\zeta\rangle)
\end{align*}
(the last equation is just \eqref{eq:antisum2}). In particular 
\begin{align}\label{eq:hilb-identity}
|\zeta|^2=\frac 14(|v|^2+|w|^2+2\langle v,w\rangle).
\end{align}
If $\zeta^*$ is the polar of $\zeta$, then 
\[
|\zeta^*|\le |\zeta|,\quad \langle\zeta^*,v\rangle+\langle\zeta,v\rangle=0,\quad \langle\zeta^*,w\rangle+\langle\zeta,w\rangle=0.
\]
Now 
\begin{align*}
    d^2(\zeta^*,-v)&=|\zeta^*|^2+|-v|^2-2\langle\zeta^*,-v\rangle\le |\zeta|^2+|v|^2-2\langle\zeta,v\rangle =|\zeta|^2-\langle v,w\rangle \\
    &=\frac 14(|v|^2+|w|^2-2\langle v,w\rangle)=\frac 14d^2(v,w)=\frac 14d^2(-v,-w),
\end{align*}
and similarly $d^2(\zeta^*,-w)\le \frac 14 d^2(-v,-w)$, implying that $\zeta^*$ is a midpoint of $[-v,-w]$. From the 2-concavity of the distance squared on $T^\omega_pY$ we obtain
\begin{align*}
|\zeta^*|^2\ge\frac{|-v|^2+|-w|^2}{2}-\frac 14d^2(-v,-w)= \frac{|v|^2+|w|^2+\langle v,w\rangle}{4}=|\zeta|^2.
\end{align*}
By Remark \ref{rmk:polar-opposite} it follows that $\zeta^*=-\zeta$ and consequently $v+w\in H_pY$. The midpoint property and the closedness of $H_pY$ in $T^\omega_pY$ further imply that $H_pY$ is a complete geodesic space. We leave it to the interested reader to verify that, equipped with this addition and the scalar multiplication from the cone structure, $H_pY$ becomes a vector space. From \eqref{eq:hilb-identity} we obtain the identity
\[
|v+w|^2=|v|^2+|w|^2+2\langle v,w\rangle,
\]
from which it readily follows that $\langle\cdot,\cdot\rangle$ is an inner product on $H_pY$.
\end{proof}

\begin{lemma}\label{lem:linearity}
    For any $\xi\in T_pY$, the map $H_pY\ni z\mapsto \langle\xi,z\rangle$ is linear.
\end{lemma}
Before the proof, we point out the following immediate corollary.
\begin{corollary}
If $\Omega\subset Y$ is open and $f:\Omega\to \R$ is locally Lipschitz and semiconcave, then for each $p\in\Omega$, there exists a unique vector $\nabla_p^Hf\in H_pY$ so that $\langle\nabla_pf,v\rangle=\langle\nabla_p^Hf,v\rangle$ for all $v\in H_pY$. In particular this $\nabla^H_pf$ satisfies the norm bound
\[
|\nabla_p^Hf|\le |\nabla_pf|\le \Lip_a f(p).
\]
\end{corollary}
In the proof below we use the notation and terminology of \cite{Alex:Kap:Pet:24}, see in particular \cite[Theorem 8.11]{Alex:Kap:Pet:24}.
\begin{proof}[Proof of Lemma \ref{lem:linearity}]
The claim is trivial if $\xi=0$ so we may assume $\xi\ne 0$. Arguing as in the proof of \cite[Lemma 13.32]{Alex:Kap:Pet:24}, let $q_n\in \operatorname{Str}(p)$ be such that $\hat\xi_n:=[p,q_n]^+\to \xi/|\xi|=:\hat \xi$ in $\Sigma_pY$.
Then $f_n:=|\xi|\dist_{q_n}$ satisfies
\begin{align*}
\ud_pf_n(v)=-|\xi|\langle\hat\xi_n,v\rangle \stackrel{n\to\infty}{\longrightarrow}-\langle \xi,v\rangle
\end{align*}
uniformly in $\Sigma_pY$ (and thus uniformly on bounded sets in $T_pY$). On the other hand, $\ud_pf_n=g_n|_{T_pY}$ where $g_n:=\ud_p^\omega f_n:T_p^\omega Y\to \R$ is a concave Lipschitz function vanishing at the origin, see Remark \ref{rmk:restr-of-concave-fn}. In particular, since $H_pY$ is a totally geodesic subspace of $T_p^\omega Y$ (cf.\ the proof above of Proposition \ref{prop:lin-subcone}) we have that $\ud_pf_n|_{H_pY}=g_n|_{H_pY}=-\langle\xi,\cdot\rangle$ is a concave Lipschitz function vanishing at the origin. Consequently the limit satisfies the same conclusions, and thus 
\[
-\langle\xi,\frac{v+w}{2}\rangle\ge \frac{-\langle\xi,v\rangle -\langle\xi,w\rangle}{2},\quad v,w\in H_pY.
\]
Replacing $v,w$ by $-v,-w$ and using $\langle\xi,- z\rangle=-\langle\xi,z\rangle$ for $z\in H_pY$ we obtain the converse inequality. Thus (using the 1-homogeneity in each argument), we obtain 
\[
\langle\xi,v+w\rangle=\langle\xi,v\rangle+\langle\xi,w\rangle,\quad v,w\in H_pY.
\]
\end{proof}
\begin{remark}{\rm
It is worth pointing out that the differential of a function might not be linear when restricted to the linear subcone,
even on Riemannian manifolds. Indeed, a concrete example occurs on the unit circle $\mathbb{S}^1$ with the usual round
metric when considering the distance function from any point. Distance increases (linearly) up to the cut locus, at which
point it starts decreasing (in both directions). Thus the differential cannot be linear.     
\fr}\end{remark}

\begin{lemma}\label{lem:nabla_f_dot_gamma}
Let \(\gamma\in\LIP([0,1];Y)\) be given. Then \(\gamma^+_t,\gamma^-_t\in\H_{\gamma_t}Y\)
exist, \(|\gamma^+_t|_{\gamma_t}=|\gamma^-_t|_{\gamma_t}=|\dot\gamma_t|\) and \(\gamma^+_t+\gamma^-_t=0\) for a.e.\ \(t\in(0,1)\).
Moreover, given an open set \(\Omega\subset Y\) containing \(\gamma([0,1])\), and a locally Lipschitz and semiconcave function \(f\colon\Omega\to\R\), it holds that
\begin{equation}\label{eq:nabla_f_dot_gamma}
(f\circ\gamma)'(t)=\langle\nabla_{\gamma_t}^\H f,\gamma^+_t\rangle\quad\text{ for a.e.\ }t\in(0,1).
\end{equation}
\end{lemma}
\begin{proof}
The first part of the statement follows from \cite[Proposition 13.9]{Alex:Kap:Pet:24} (see also \cite[proof of Exercise 13.27 at page 270]{Alex:Kap:Pet:24}).
The second part follows from \cite[Exercise 13.27]{Alex:Kap:Pet:24}.
\end{proof}

Let $HX=\{(x,v):\ x\in X,\ v\in H_xX\}$ be the \emph{linear bundle} of $X$, i.e.\ for every $x\in X$, $H_xX$ is
the linear subcone of the tangent space $T_xX$, see Section \ref{sec:alex}. A Borel section of $HX$ is a map $v:X\to HX$
such that $v_x\in H_xX$ for every $x\in X$ and for which
\[
x\mapsto \langle v_x,\nabla^H_x\dist_y^2\rangle_x
\]
is Borel on \(X\setminus\{y\}\) for every $y\in X$.
\begin{remark}{\rm
We could alternatively define a $\sigma$-algebra on $HX$ as in \cite[Section 3]{DiMar:Gig:Pas:Sou:20},
and the ensuing notion of Borel sections would coincide. The verification of this is left to the interested reader,
as we will not use this fact in what follows.
\fr}\end{remark}

\begin{definition}
The \emph{linear $L^2$-tangent module} $\Gamma_\mu(HX)$ consists of Borel sections $v:X\to HX$ such that
$x\mapsto |v_x|_x$ belongs to $L^2(\mu)$. We equip $\Gamma_\mu(HX)$ with the norm
\[
\|v\|_{\Gamma_\mu(HX)}:=\left(\int_X|v_x|_x^2\,\d\mu(x)\right)^{1/2}.
\]
\end{definition}
\begin{remark}{\rm
Since $(H_xX,|\cdot|_x)$ is a Hilbert space for all $x\in X$, we get that $\|\cdot\|_{\Gamma_\mu(HX)}$ is an inner product
norm. It is straighforward to see that this norm is complete, so that $(\Gamma_\mu(HX), \|\cdot\|_{\Gamma_\mu(HX)})$ is a
Hilbert space. It is easy to check that \(\Gamma_\mu(HX)\) is also an \(L^2(\mu)\)-normed \(L^2(\mu)\)-module.
\fr}\end{remark}

\subsection{Isometric embedding \texorpdfstring{$L^2(TX)\hookrightarrow\Gamma_\mu(HX)$}{L2-in-L2}}
We are now in a position to get Theorem \ref{thm:alex-geom-bundle}, as a consequence of Theorem \ref{thm:der-rep-vectorfield} below.
Before that, let us observe the following fact. Assume that \((X,d)\) is an Alexandrov space containing at least two points. Given any
boundedly-finite Radon measure \(\mu\) on \(X\), we claim that we can find a set \(D\subset{\rm spt}(\mu)\) such that
\begin{equation}\label{eq:hp_D}
D\text{ is a countable dense subset of }{\rm spt}(\mu)\text{ with }\mu(\{y\})=0\text{ for every }y\in D.
\end{equation}
To prove this, fix a countable dense set \(\tilde D\subset{\rm spt}(\mu)\). For any \(z\in\tilde D\), we can find a sequence of points \((y^z_n)_n\subset X\)
converging to \(z\) such that \(\mu(\{y^z_n\})=0\) for all \(n\in\N\) (because for any constant-speed Lipschitz curve \(\gamma\colon[0,1]\to X\) joining
\(z\) and some other point of \(X\), we have that \(\mu(\{\gamma(t)\})=0\) for all but countably many \(t\in[0,1]\) due to the \(\sigma\)-finiteness of \(\mu\)).
Hence, \(D\coloneqq\{y^z_n:z\in\tilde D,\,n\in\N\}\) fulfills the requirements of \eqref{eq:hp_D}.
\begin{theorem}\label{thm:der-rep-vectorfield}
Let \((X,d)\) be an Alexandrov space. Let $\mu$ be a boundedly-finite Radon measure on $X$. Fix any derivation \(b\in\Der^2_2(X,\mu)\)
such that the support of \(|b|\) is bounded. Then there exists a unique vector field \(v_b\in\Gamma_\mu(HX)\) such that
for any \(y\in X\) it holds that
\begin{equation}\label{eq:ineq_v_pi_b_pi}
b(\dist_y)(x)=\langle\nabla_x^H\dist_y,v_b(x)\rangle\quad\text{ for }\mu\text{-a.e.\ }x\in X\setminus\{y\}.
\end{equation}
Moreover, it holds that
\begin{equation}\label{eq:ineq_norm_v_pi}
|v_b(x)|_x=|b|(x)\quad\text{ for }\mu\text{-a.e.\ }x\in X.
\end{equation}
\end{theorem}
Since $v_b$ is uniquely determined by \eqref{eq:ineq_v_pi_b_pi}, we have in particular that $b\mapsto v_b$ is linear.
\begin{proof}
If \(X\) is a singleton, then \({\rm Der}^2_2(X,\mu)=\{0\}\) and thus the statement trivially holds. Thus, assume
\(X\) contains at least two points. By the Paolini--Stepanov superposition principle \cite{Pao:Ste:12,Pao:Ste:13},
we find a finite Borel measure \(\ppi\) on \(C([0,1];{\rm spt}(\mu))\), concentrated on non-constant Lipschitz curves
with constant speed, such that
\begin{subequations}\begin{align}
\label{eq:PS_1}
\int g\,b(f)\,\d\mu&=\int\!\!\!\int_0^1 g(\gamma_t)(f\circ\gamma)'_t\,\d t\,\d\ppi(\gamma)\quad\text{ for every }g\in\LIP_{bs}(X)\text{ and }f\in\LIP(X),\\
\label{eq:PS_2}
\int g|b|\,\d\mu&=\int\!\!\!\int_0^1 g(\gamma_t)|\dot\gamma_t|\,\d t\,\d\ppi(\gamma)\quad\text{ for every }g\in\LIP_{bs}(X);
\end{align}\end{subequations}
see \cite[Theorem 4.9 and Lemma 6.1]{DiMar:Gig:Pas:Sou:20}; here, we use the fact that \(|b|,\div(b)\in L^1(\mu)\)
(which holds because the supports of \(|b|\), \(\div(b)\) are bounded and \(\mu\) is boundedly finite).
Now, consider the disintegration \(\hat\ppi=\int\hat\ppi_x\,\d(\e_\#\ppi)\)
of \(\hat\ppi\coloneqq\ppi\otimes\mathscr L^1|_{[0,1]}\) along the evaluation map \(\e(\gamma,t)\coloneqq\gamma_t\). The properties
of \(\ppi\) ensure that \(\e_\#\hat\ppi\ll\mu\), thus it makes sense to define \(\rho\coloneqq\frac{\d(\e_\#\hat\ppi)}{\d\mu}\). Next, we define
\[
v_b(x)\coloneqq\rho(x)\int\gamma^+_t\,\d\hat\ppi_x(\gamma,t)\in H_x X\quad\text{ for }\mu\text{-a.e.\ }x\in X.
\]
Note that \(v_b\in\Gamma_\mu(HX)\). Fix \(y\in X\) and \(g\in\LIP_{bs}(X\setminus\{y\})\).
Using \eqref{eq:PS_1} and \eqref{eq:nabla_f_dot_gamma}, we compute
\[\begin{split}
\int g\,b(\dist_y)\,\d\mu&=\int\!\!\!\int_0^1 g(\gamma_t)(\dist_y\circ\gamma)'_t\,\d t\,\d\ppi(\gamma)=\int g(\gamma_t)\langle\nabla_{\gamma_t}^\H\dist_y,\gamma^+_t\rangle\,\d\hat\ppi(\gamma,t)\\
&=\int g(x)\rho(x)\int\langle\nabla_x^\H\dist_y,\gamma^+_t\rangle\,\d\hat\ppi_x(\gamma,t)\,\d\mu(x)=\int g(x)\langle\nabla_x^\H\dist_y,v_b(x)\rangle\,\d\mu(x),
\end{split}\]
whence \eqref{eq:ineq_v_pi_b_pi} follows. Moreover, using \eqref{eq:PS_2} and Lemma \ref{lem:nabla_f_dot_gamma}, we can estimate
\[\begin{split}
\int g(x)|v_b(x)|_x\,\d\mu(x)&=\int g(x)\rho(x)\bigg|\int\gamma^+_t\,\d\hat\ppi_x(\gamma,t)\bigg|_x\,\d\mu(x)
\leq\int g(x)\rho(x)\int|\dot\gamma_t|\,\d\hat\ppi_x(\gamma,t)\,\d\mu(x)\\
&=\int\!\!\!\int_0^1 g(\gamma_t)|\dot\gamma_t|\,\d t\,\d\ppi(\gamma)=\int g|b|\,\d\mu,
\end{split}\]
implying $|v_b|\le |b|$. To prove the converse inequality, take \(D\) as in \eqref{eq:hp_D}. For any \(y\in D\) we have
\begin{align*}
b(\dist_y)=\langle \nabla^H\dist_y,v_b\rangle\le\Lip(\dist_y)|v_b|\le |v_b| \quad \text{ in the }\mu\text{-a.e.\ sense,}
\end{align*}
whence it follows that $|b|=\bigvee_{y\in D}b(\dist_y)\le |v_b|$ thanks to \eqref{eq:dist_y_norming}.
All in all, \eqref{eq:ineq_norm_v_pi} is proved. Finally, we show that \eqref{eq:ineq_v_pi_b_pi} determines $v_b$ uniquely. Indeed, suppose that $b\ne \tilde b$ but $v_b=v_{\tilde b}$ $\mu$-a.e. Then, by \eqref{eq:dist_y_norming}, there exists $y\in Y$ and a Borel set $A\subset X$ with $\mu(A)>0=\mu(\{y\})$ so that $|b(\dist_y)-\tilde b(\dist_y)|>0$ on $A$. However \eqref{eq:ineq_v_pi_b_pi} and the fact that $v_b=v_{\tilde b}$ yield $b(\dist_y)(x)=\langle v_b(x),\nabla_x^H\dist_y\rangle=\langle v_{\tilde b}(x),\nabla_x^H\dist_y\rangle=\tilde b(\dist_y)$ for $\mu$-a.e.\ $x\in A\setminus\{y\}$, a contradiction.
\end{proof}

\begin{proposition}\label{prop:Der_Hilb}
Let \((X,d)\) be an Alexandrov space. Let \(\mu\) be a boundedly-finite Radon measure on \(X\). Then it holds that
\begin{equation}\label{eq:parall_Alex}
|b+\tilde b|^2+|b-\tilde b|^2=2|b|^2+2|\tilde b|^2\quad\text{ for every }b,\tilde b\in\Der^2_2(X,\mu).
\end{equation}
\end{proposition}
\begin{proof}
First, fix any \(b,\tilde b\in\Der^2_2(X,\mu)\) such that \(|b|,|\tilde b|\) have bounded support.
Take \(D\) as in \eqref{eq:hp_D}. Given any \(\alpha,\beta\in[0,1]\) with \(\alpha^2+\beta^2\leq 1\)
and \(y,z\in D\), we deduce from \eqref{eq:ineq_v_pi_b_pi} and \eqref{eq:ineq_norm_v_pi} that
\[\begin{split}
&(b+\tilde b)(\alpha\,\dist_y)(x)+(b-\tilde b)(\beta\,\dist_z)(x)\\
=\,&\alpha\,b(\dist_y)(x)+\alpha\,\tilde b(\dist_y)(x)+\beta\,b(\dist_z)(x)-\beta\,\tilde b(\dist_z)(x)\\
=\,&\alpha\langle\nabla_x^H\dist_y,v_b(x)\rangle+\alpha\langle\nabla_x^H\dist_y,v_{\tilde b}(x)\rangle+\beta\langle\nabla_x^H\dist_z,v_b(x)\rangle
-\beta\langle\nabla_x^H\dist_z,v_{\tilde b}(x)\rangle\\
=\,&\big\langle\alpha\,\nabla_x^H\dist_y+\beta\,\nabla_x^H\dist_z,v_b(x)\big\rangle+\big\langle\alpha\,\nabla_x^H\dist_y-\beta\,\nabla_x^H\dist_z,v_{\tilde b}(x)\big\rangle\\
\leq\,&\sqrt{\big|\alpha\,\nabla_x^H\dist_y+\beta\,\nabla_x^H\dist_z\big|_x^2+\big|\alpha\,\nabla_x^H\dist_y-\beta\,\nabla_x^H\dist_z\big|_x^2}\,\sqrt{|v_b(x)|_x^2+|v_{\tilde b}(x)|_x^2}\\
\leq\,&\sqrt{2\big|\alpha\,\nabla_x^H\dist_y\big|_x^2+2\big|\beta\,\nabla_x^H\dist_z\big|_x^2}\,\sqrt{|b|^2(x)+|\tilde b|^2(x)}
\leq\sqrt{\alpha^2+\beta^2}\sqrt{2|b|^2(x)+2|\tilde b|^2(x)}\\
\leq\,&\sqrt{2|b|^2(x)+2|\tilde b|^2(x)}
\end{split}\]
for \(\mu\)-a.e.\ \(x\in X\). Thanks to Lemma \ref{lem:formula_|b|^2+|tilde_b|^2} and \eqref{eq:dist_y_norming},
by taking the supremum over \(\alpha\), \(\beta\), \(y\), \(z\) we get
\[
\sqrt{|b+\tilde b|^2+|b-\tilde b|^2}\leq\sqrt{2|b|^2+2|\tilde b|^2}.
\]
The converse inequality can be obtained by replacing \(b\) and \(\tilde b\) with \(\frac{b+\tilde b}{2}\) and \(\frac{b-\tilde b}{2}\), respectively.
All in all, \eqref{eq:parall_Alex} is proved when \(|b|,|\tilde b|\) have bounded support. The general case then follows
by an approximation argument, by taking Remark \ref{rmk:bdd_spt_der_dense} into account.
\end{proof}
%
%

We are now ready to complete the proof of Theorem \ref{thm:alex-geom-bundle}:
\begin{proof}[Proof of Theorem \ref{thm:alex-geom-bundle}]
By virtue of Remark \ref{rmk:bdd_spt_der_dense}, the map \(b\mapsto v_b\) we constructed in Theorem \ref{thm:suff_condit_IH}
can be uniquely extended to a linear pointwise-norm-preserving map \(\iota\colon L^2_\Lip(TX)\hookrightarrow\Gamma_\mu(HX)\).
Given \(h\in\LIP_{bs}(X)\), \(b\in\Der^2_2(X,\mu)\) such that the support of \(|b|\) is bounded, and \(y\in X\), we have
\[
\langle\nabla^H_x\dist_y,h(x)v_b(x)\rangle=h(x)\,b(\dist_y)(x)=(hb)(\dist_y)(x)=\langle\nabla^H_x\dist_y,v_{hb}(x)\rangle
\]
for \(\mu\)-a.e.\ \(x\in X\setminus\{y\}\) by \eqref{eq:ineq_v_pi_b_pi}. By the uniqueness of \(v_{hb}\), we deduce that
\(\iota(hb)=v_{hb}=h v_b=h\,\iota(b)\). Using Remark \ref{rmk:bdd_spt_der_dense} and the fact that \(\LIP_{bs}(X)\) is weakly\(^*\)
dense in \(L^\infty(\mu)\), by approximation we conclude that \(\iota\colon L^2_\Lip(TX)\hookrightarrow\Gamma_\mu(HX)\)
is \(L^\infty(\mu)\)-linear. Finally, as the infinitesimal Hilbertianity of \((X,\mu)\) and
\eqref{eq:inf_Hilb_gives_Hilb_Liptgmod} ensure that \(L^2_\Lip(TX)\) and \(L^2(TX)\) can be identified, the claim follows.
\end{proof}

\section{Infinitesimal splitting and universal infinitesimal Hilbertianity}\label{sec:inf-split}

In this section, we prove Theorem \ref{thm:inf-split-implies-inf-hilb} and Corollary \ref{cor:rcd-alex-inf-hilb}. Recall that a line $\ell\subset Z$ in a metric space $Z$ is an isometric image of $\R$. We say that $\ell\subset Z$ is a \emph{splitting line} if there exists an isometric homeomorphism $\iota:Z\to Y\times \R$ onto a direct product $Y\times \R$ such that $P_Y(\iota(\ell))$ is a point, where $P_Y$ is the projection to $Y$. Moreover, a metric space $Z$ is said to have the \emph{splitting property} if every line in $Z$ is a splitting line. In the following definition, $\operatorname{Tan}(X,x)$ refers to the union of the collections of Gromov--Hausdorff tangents and ultratangents of $X$ at $x$.
\begin{definition}\label{def:splitting}
A metric space $X$ is
\begin{itemize}
    \item[(i)] infinitesimally splitting if, for any $x\in X$, every $(Y,o)\in \operatorname{Tan}(X,x)$ has the splitting property;
    \item[(ii)] infinitesimally curve-splitting if, for any Lipschitz curve $\gamma$ in $X$ and a.e. $t$, any blow-up $\gamma_\infty:\R\to Y$ of $\gamma$ at $t$ with $\gamma_\infty(0)=0$ is a splitting line;
    \item[(iii)] \emph{weakly} infinitesimally (curve-)splitting if, for every $x\in X$, \emph{there exists} $(Y_x,o_x)\in \operatorname{Tan}(X,x)$ so that condition (i) (resp. (ii)) is satisfied for $(Y_x,o_x)$ instead of all $(Y,o)\in\operatorname{Tan}(X,x)$.
\end{itemize}
\end{definition}
Clearly, (weak) infinitesimal splitting implies (weak) infinitesimal curve-splitting, since the latter only requires certain geodesics to split an $\R$-factor while the former requires all geodesics to split.

Different variants of Definition \ref{def:splitting} can be obtained by choosing $\operatorname{Tan}$ to refer to Gromov--Hausdorff tangents or to ultratangents (instead of the union of the two). While ultratangents are always guaranteed to exist, Gromov--Hausdorff tangents (when they exist) are better behaved in some ways. However, even when both exist (for a given sequence of scales), they need not coincide (e.g. if $X$ is infinite dimensional), and there is no obvious relationship between the splitting properties of the respective spaces. For our purposes any choice of tangent is sufficient, and we chose the present formulation because it gives the weakest variant of weak infinitesimal (curve-)splitting.

\begin{remark}\label{rmk:inf-dim-case}{\rm
Finite dimensional RCD-spaces (and thus finite dimensional Alexandrov spaces) are infinitesimally splitting (in this case, the Gromov--Hausdorff tangents and ultratangents coincide and are finite dimensional non-negatively curved RCD-spaces).

By contrast, splitting (and thus the proof strategy for Theorem \ref{thm:inf-split-implies-inf-hilb}) is known to fail for general RCD($0,\infty$)-spaces. In particular, our results do not cover the case of RCD($K,\infty$)-spaces and we do not know whether they are universally infinitesimally Hilbertian. For infinite dimensional Alexandrov spaces the tangent cone is not necessarily a Gromov--Hausdorff tangent and we do not know whether it is splitting. However, any ultratangent of an Alexandrov space is an Alexandrov space with non-negative curvature, and thus satisfies the splitting property.
\fr}\end{remark}

The proof of Theorem \ref{thm:inf-split-implies-inf-hilb} relies on blow-up analysis. Below we formulate the key proposition connecting infinitesimal splitting and Hilbertianity.

\begin{proposition}\label{prop:pre-parall-rule}
Suppose $X$ is weakly infinitesimally curve-splitting, and $\mu$ is a boundedly finite Radon measure on $X$. Then, for any $f,g\in\LIP(X)$ we have
\[
|D(f+g)|^2+|D(f-g)|^2\le 2(\Lip_af)^2+2(\Lip_ag)^2
\]
$\mu$-a.e. on $X$.
\end{proposition}

We remark that the proof remains true even if $X$ is weakly infinitesimally splitting along all curves outside an exceptional curve family $\Gamma_0$ with $\operatorname{Mod}_2(\Gamma_0;\mu)=0$, because the curves $\gamma,\sigma\in \mathcal C_x$ can be chosen not to lie in $\Gamma_0$ for $\mu$-a.e. $x\in X$.
\begin{proof}
Let $\varphi=(f,g)\in \LIP(X,\R^2)$. For $x\in X$, denote by $\mathcal C_x$ the collection of Lipschitz curves $\gamma:(-\delta,\delta)\to X$ with $\gamma_0=x$ which are metrically differentiable at 0 and for which $\varphi\circ\gamma$ is differentiable at 0. Observe that, for each $\gamma\in \mathcal C_x$, we have $\md_0\gamma(1)=|\gamma_0'|$ and that, for any $(Y,o)\in \operatorname{Tan}(X,x)$, the blow-up $\gamma_\infty:\R\to Y$ of $\gamma$ at 0 exists and satisfies
\begin{align*}
\gamma_\infty(0)=o,\quad d_Y(\gamma_\infty(t),\gamma_\infty(s))=|\gamma_0'||t-s|,\quad s,t\in\R.
\end{align*}
Similarly, $(\varphi\circ\gamma)'(0)$ exists for each $\gamma\in \mathcal C_x$. Moreover, if $r_j$ is a sequence of scales so that $(Y,o)$ is the GH or $\omega$-limit of $(r_j\inv X,x)$, up to a subsequence (in the case of $\omega$-convergence no passage to a subsequence is needed) $\varphi_j=r_j\inv(\varphi-\varphi(x))$ GH- or $\omega$-converge to a map $\psi=(f_\infty,g_\infty)\in\LIP(X,\R^2)$ with $\psi(o)=0$, and it follows that
\begin{align*}
\psi(\gamma_\infty(t))=\lim \frac{\varphi(\gamma(r_jt))-\varphi(\gamma_0)}{r_j}=t(\varphi\circ\gamma)'(0),\quad t\in \R.
\end{align*}
The limit above is either an $\omega$-limit or the limit as $j\to\infty$ along a subsequence of $r_j$ (depenging on $\varphi$ but not on $\gamma$ or $t$).

By \cite[Theorem 1.1]{EB-sou24} there exists a 2-plan $\bm\eta$ on $X$ such that the disintegration $\{\bar{\bm\eta}_x\}$ of $|\gamma_t'|\ud t\ud\bm\eta$ with respect to the map $(\gamma,t)\mapsto \gamma_t$ satisfies
\[
|D(\xi\circ\varphi)|(x)=\left\|\frac{\xi((\varphi\circ\gamma)'(t))}{|\gamma_t'|} \right\|_{L^\infty(\bar{\bm\eta}_x)}\quad\text{for all }\xi\in (\R^2)^*
\]
for $\mu$-a.e. $x\in \{|D\varphi|\ne 0\}$. It follows that $\mathcal C_x$ is non-empty for $\mu$-a.e. $x\in\{|D\varphi|\ne 0\}$. The claim in the statement of the proposition is clearly true $\mu$-a.e. $x\in \{|D\varphi|=0\}$. For the remainder of the proof we fix $x\in \{|D\varphi|\ne 0\}$ so that all the considerations above hold.

Let $\varepsilon>0$ and suppose $\gamma,\sigma\in \mathcal C_x$ are such that $|\gamma_0'|=|\sigma_0'|=1$, $\gamma_\infty$ and $\sigma_\infty$ are splitting lines in $(Y,o)$, and
\begin{align*}
    (1-\varepsilon)|D(f+g)|(x)\le ((f+g)\circ\gamma)'(0),\quad (1-\varepsilon)|D(f-g)|(x)\le ((f-g)\circ\sigma)'(0).
\end{align*}
By \cite[Theorem 1.1]{Lytchak-Teri-Thomas} we have that $Y=Z\times H$, $o=(z_0,0)$, where $H$ is a 2-plane containing the images of $\gamma_\infty,\sigma_\infty$ (or a line, if the images coincide). In particular $H=\operatorname{span}\{\gamma_\infty(1),\sigma_\infty(1)\}$, and the geodesic maps $\gamma_\infty,\sigma_\infty:\R\to H$ are linear. We define linear maps $F,G:H\to \R$ by setting 
\begin{align*}
F(\gamma_\infty(a)+\sigma_\infty(b)):=a(f\circ\gamma)'_0+b(f\circ\sigma)'_0,\quad G(\gamma_\infty(a)+\sigma_\infty(b)):=a(g\circ\gamma)'_0+b(g\circ\sigma)'_0
\end{align*}
for all $a,b\in\R$. Then we have
\begin{align*}
(1-\varepsilon)[a|D(f+g)|(x)+b|D(f-g)|(x)]\le & a((f+g)\circ\gamma)_0'+b((f-g)\circ\sigma)_0'\\
=&a(f\circ\gamma)'_0+b(f\circ\sigma)'_0+a(g\circ\gamma)_0'-b(g\circ\gamma)'_0\\
=&F(\gamma_\infty(a)+\sigma_\infty(b))+G(\gamma_\infty(a)-\sigma_\infty(b)).
\end{align*}
If $\|A\|$ denotes the operator norm of a linear map $A:H\to \R$, the last expression is at most
\begin{align*}
&\le \|F\|\|\gamma_\infty(a)+\sigma_\infty(b)\|_H+\|G\|\|\gamma_\infty(a)-\sigma_\infty(b)\|_H\\
&\le (\|F\|^2+\|G\|^2)^{1/2}(\|\gamma_\infty(a)+\sigma_\infty(b)\|_H^2+\|\gamma_\infty(a)-\sigma_\infty(b)\|_H^2)^{1/2}\\
&=(\|F\|^2+\|G\|^2)^{1/2}(2\|\gamma_\infty(a)\|_H^2+2\|\sigma_\infty(b)\|_H^2)^{1/2)}=(\|F\|^2+\|G\|^2)^{1/2}(2a^2+2b^2)^{1/2}.
\end{align*}

Taking supremum over all $a,b\in\R^2$ with $a^2+b^2\le 1$ we arrive at 
\begin{align}\label{eq:est-with-op-norm}
(1-\varepsilon)(|D(f+g)|(x)^2+|D(f-g)|(x)^2)^{1/2}\le (2\|F\|^2+2\|G\|^2)^{1/2}.
\end{align}
But
\begin{align}\label{eq:est-of-op-norm}
\|F\|\le \Lip_af(x),\quad \|G\|\le \Lip_ag(x);
\end{align}
Indeed, for all $\gamma_\infty(a)+\sigma_\infty(b)\in B_H(0,1)$ we have that 
\begin{align*}
F(\gamma_\infty(a)+\sigma_\infty(b))=f_\infty(\gamma_\infty(a))+f_\infty(\sigma_\infty(b))=f_\infty(\gamma_\infty(a))-f_\infty(-\sigma_\infty(b))\\
\le \LIP(f_\infty)\|\gamma_\infty(a)-(-\sigma_\infty(b))\|_H\le \LIP(f_\infty)\le \Lip_af(x),
\end{align*}
and similarly for $G$. Together \eqref{eq:est-with-op-norm} and \eqref{eq:est-of-op-norm} yield
\[
(1-\varepsilon)^2(|D(f+g)|(x)^2+|D(f-g)|(x)^2)\le 2\Lip_af(x)^2+2\Lip_ag(x)^2.
\]
Since $\varepsilon$ can be taken arbitrarily small, the claim follows. 
\end{proof}

\begin{remark}\label{rmk:no-linear}{\rm
By \cite[Theorem 1.1]{Lytchak-Teri-Thomas}, any tangent $(Y,o)\in \operatorname{Tan}(X,x)$ can be written as a product $Y=Z\times H_{max}$, $o=(z_0,0)$, where $H_{max}$ is the maximal Hilbert space factor and $Z$ does not admit any splitting lines (through $z_0$). One may ask whether, given a blow-up $\psi:Y\to \R^k$ of a Lipschitz map $\varphi:X\to \R^k$ at $x$, the restriction $\psi(z_0,\cdot):H_{max}\to \R^k$ is linear. This need not be the case, cf.\ \cite{marchese-schioppa}.
\fr}\end{remark}

\begin{proof}[Proof of Theorem \ref{thm:inf-split-implies-inf-hilb}]
Let $\mu$ be a boundedly finite Radon measure on $X$. By Proposition \ref{prop:pre-parall-rule} we have
\begin{align}\label{eq:pre-parall-rule}
\int|D(f+g)|^2\ud\mu+\int|D(f-g)|^2\ud\mu\le 2\int(\Lip_af)^2\ud\mu+2\int(\Lip_ag)^2\ud\mu
\end{align}
for all $f,g\in\LIP_b(X)$. For $f,g\in W^{1,2}(X,\mu)$, using the energy density of Lipschitz functions \cite{Amb:Gig:Sav:13} (see also \cite{EB:20:published}), we obtain sequences $(f_j),(g_j)\subset\LIP_b(X)$ with $f_j\to f$, $g_j\to g$ and $\Lip_af_j\to |Df|$, $\Lip_ag_j\to |Dg|$ in $L^2(\mu)$. By \eqref{eq:pre-parall-rule} and the lower semicontinuity of the Cheeger energy we obtain 
\begin{align*}
    \int |D(f+g)|^2\,\d\mu+\int|D(f-g)|^2\,\d\mu\le &\liminf_{j\to\infty}\left[\int|D(f_j+g_j)|^2\,\d\mu+\int|D(f_j-g_j)|^2\,\d\mu\right]\\
    \le &2\liminf_{j\to\infty}\left[\int(\Lip_af_j)^2\,\d\mu+\int(\Lip_ag_j)^2\,\d\mu\right] \\
    = &2\int|Df|^2\,\d\mu+2\int|Dg|^2\,\d\mu.
\end{align*}
This readily implies the parallelogram identity, showing that $W^{1,2}(X,\mu)$ is a Hilbert space.
\end{proof}

\begin{proof}[Proof of Corollary \ref{cor:rcd-alex-inf-hilb}]
    If $X=(X,\mathfrak m)$ is an ${\rm RCD}(K,N)$-space for some $K\in\R,\ N\in [1,\infty)$, then for any $x\in X$, every $(Y,o)\in \operatorname{Tan}_{GH}(X,x)$ is an RCD($0,N)$-space by the stability of the RCD($K,N$)-condition under pmGH-convergence and the scaling properties of the RCD($K,N$)-condition under scaling of the metric (see e.g. \cite[Section 2, p. 471]{Brue:Pas:Sem:21}). Thus Gigli's splitting theorem \cite{Gig:14,Gig:13} implies that every $(Y,o)\in \operatorname{Tan}_{GH}(X,x)$ has the splitting property, i.e.\ $X$ is infinitesimally splitting (for Gromov--Hausdorff tangents). Moreover $\operatorname{Tan}_{GH}(X,x)\ne \varnothing$ for all $x$, since $X$ is locally doubling. By Theorem \ref{thm:inf-split-implies-inf-hilb}, $X$ is universally infinitesimally Hilbertian. 
    
    If $X$ is an Alexandrov space and $x\in X$, any ultratangent $(Y,o)\in \operatorname{Tan}(X,x)$ is a non-negatively curved Alexandrov space (see \cite[Theorem 6.17, Theorem 13.1 and Observation 4.9]{Alex:Kap:Pet:24}). Thus by \cite[Theorem 16.22]{Alex:Kap:Pet:24} $Y$ has the splitting property and, by Theorem \ref{thm:inf-split-implies-inf-hilb}, $X$ is universally infinitesimally Hilbertian.
\end{proof}

\appendix

\section{Universal infinitesimal Hilbertianity and property (ET)}
Recall that a metric space $X$ has property (ET) if, for any $u\in W^{1,2}(\mathbb D,X)$, the approximate metric derivative $\md_z u$ of $u$ is a (possibly degenerate) inner product norm for a.e.\ $z\in \mathbb D$, cf.\ \cite[Definition 11.1]{Lyt:Wen:17}. We refer to \cite[Section 1.2]{Lyt:Wen:17} for the definition of $W^{1,2}(\mathbb D,X)$ and the approximate metric derivative.

\begin{theorem}\label{thm:UIH-implies-ET}
    Suppose $X$ is complete and universally infinitesimally Hilbertian. Then $X$ has property (ET).
\end{theorem}

We record two elementary lemmas needed in the proof of Theorem \ref{thm:UIH-implies-ET}.
\begin{lemma}\label{lem:mod-est}
Suppose $B\subset \R^2$ is a ball, $u\in W^{1,2}(B,X)$ and $\mu:=u_\ast(|Du|^2\mathcal L^2|_B)$. Then
\[
\operatorname{Mod}_2(\Gamma)\le \operatorname{Mod}_2(u(\Gamma);\mu)
\]
for all curve families $\Gamma$ in $B$. 
\end{lemma}
Here $u(\Gamma)$ denotes the collection of rectifiable curves $u\circ\gamma$ with $\gamma\in \Gamma$.
\begin{proof}
Let $\rho$ be admissible for $u(\Gamma)$ and suppose $|(u\circ\gamma)_t'|\le |Du|(\gamma_t)|\gamma_t'|$ a.e.\ $t$ for all $\gamma\notin \Gamma_0$, where $\operatorname{Mod}_2(\Gamma_0)=0$. Then
\begin{align*}
1\le \int_{0}^1\rho(u\circ\gamma(t))|(u\circ\gamma)_t'|\ud t\le \int_0^1\rho\circ u(\gamma_t)|Du|(\gamma_t)|\gamma_t'|\ud t
\end{align*}
for all $\gamma\in \Gamma\setminus\Gamma_0$, so that $\rho\circ u|Du|$ is admissible for $\Gamma\setminus\Gamma_0$. Thus
\[
\operatorname{Mod}_2(\Gamma)=\operatorname{Mod}_2(\Gamma\setminus\Gamma_0)\le \int_B\rho^2\circ u|Du|^2\ud z=\int_X\rho^2\ud\mu.
\]
Taking infimum over admissible $\rho$ yields the claim.
\end{proof}

For the terminology used in the statement of the second lemma we refer the reader to \cite[Section 4B]{EB-sou24} for the \emph{canonical minimal gradient} $\Phi:(\R^n)^*\times X\to [0,\infty]$ of $\varphi\in W^{1,2}((X,\mu);\R^n)$. We recall here that $\Phi$ is a Borel function such that (i) $\Phi_\xi:=\Phi(\xi,\cdot)$ is a representative of the minimal 2-weak upper gradient $|D(\xi\circ \varphi)|$ for each $\xi\in (\R^n)^*$, and (ii) $\Phi^x:=\Phi(\cdot,x)$ is a seminorm on $(\R^n)^*$ for $\mu$-a.e.\ $x$. 

\begin{lemma}\label{lem:IH-inner prod}
Suppose $(X,\mu)$ is infinitesimally Hilbertian. Then for any $\varphi\in W^{1,2}(X,\R^2)$, the canonical minimal gradient $\Phi$ associated to $\varphi$ has the following property: for $\mu$-a.e.\ $x\in X$, $\Phi^x$ is an inner product seminorm.
\end{lemma}
\begin{proof}
For $\xi,\zeta\in (\R^2)^*$ and $\mu$-a.e.\ $x$ we have
\begin{align*}
\Phi^x(\xi+\zeta)^2+\Phi^x(\xi-\zeta)^2&=|D((\xi+\zeta)\circ \varphi)|^2(x)+|D((\xi-\zeta)\circ \varphi)|^2(x)\\
&=2(|D(\xi\circ\varphi)|^2(x)+|D(\zeta\circ\varphi)|^2(x))\\
&=2(\Phi^x(\xi)^2+\Phi^x(\zeta)^2).
\end{align*}
Applying this a.e.\ identity for a countable dense set of $\xi,\zeta$ and using the fact that $\Phi^x$ is a seminorm the claim follows.
\end{proof}

\begin{proof}[Proof of Theorem \ref{thm:UIH-implies-ET}]
Let $u\in W^{1,2}(\mathbb D,X)$ be such that $s_z:=\md_z u$ fails to be an inner product seminorm on a set of positive measure. If $s_z$ is degenerate for a.e.\ $z\in \mathbb D$, i.e.\ $\{0\}\ne \ker(s_z):=\{v\in \R^2:\ s_z(v)=0\}$ for a.e.\ $z\in \mathbb D$, then (as $\ker(s_z)$ is a vector space) $s_z$ is a norm on the quotient
\begin{align*}
\R^2/\ker(s_z)=
\left\{
\begin{array}{ll}
\R\\
\{0\}
\end{array}\quad\begin{array}{ll}
\textrm{if }\dim\ker(s_z)=1\\
\textrm{if }\dim\ker(s_z)=2.
\end{array}
\right.
\end{align*}
In either case $s_z$ is an inner product norm on the quotient, and thus $s_z$ is a degenerate inner product norm on $\R^2$. 

Thus we may assume that there is a compact set $K'\subset \mathbb D$ with $|K'|>0$ such that $\md_zu$ exists and is a non-degenerate non-inner product norm on $\R^2$ for every $z\in K'$. It follows that there exists a further compact set $K\subset K'$ with $|K|>0$ such that $u|_K$ is $L$-bi-Lipschitz for some $L>0$.

Let $z\in K$ be a density point of $K$ and $r_0>0$ such that $|K\cap B(z,r)|\ge \frac 12|B(z,r)|$, $r\le r_0$. Set $\mu:=u_\ast(|Du|^2\mathcal L^2|_{B(z,r_0)})$. Denote $K_0:=K\cap B(z,r_0)$ and write
\[
\mu=\rho\mathcal H^2|_{u(K_0)}+\nu,\quad \nu\perp \mathcal H^2|_{u(K_0)}.
\]
Since $u|_{K_0}$ is $L$-bi-Lipschitz we have $\rho\ge L^{-4}$ on $u(K_0)$.

Let $U\subset u(K_0)$ be a compact set with $\mathcal H^2(U)>0$ and $\nu(U)=0$. (Note in particular that $\mu|_U$ and $\mathcal H^2|_U$ are mutually absolutely continuous.)
Let $\varphi\in\LIP(X,\R^2)$ be a $\sqrt 2L$-Lipschitz extension of $u\inv:u(K_0)\to K_0\subset \mathbb D\subset \R^2$. We claim that $(U,\varphi)$ is 2-independent in the sense of \cite[Section 1C]{EB-sou24} in the metric measure space $(X,\mu)$. We denote by $\Phi(\xi,x)$ the \emph{canonical minimal gradient} associated to $\varphi$, see \cite[Section 4B]{EB-sou24}.

Firstly, for $\operatorname{Mod}_2$-a.e.\ curve $\gamma$ in $(X,\mu)$ we have that $\md_{\varphi(\gamma_t)}u$ exists and equals $\md_{\varphi(\gamma_t)}(u|_{K_0})$ for a.e.\ $t\in \gamma\inv(U)$. Indeed, for any null-set $N\subset  \mathbb D$, the curve family $$\Gamma_1:=\{ \gamma:\ |\gamma\inv(U)\cap(\varphi\circ\gamma)\inv(N)|>0\}$$ has zero 2-modulus in $(X,\mu)$: if $\gamma\in \Gamma_1$ is constant-speed parametrized (reparametrizing the curves in a family does not affect its 2-modulus) then 
\[
\int_\gamma\chi_{U\cap \varphi\inv(N)}\ud s=\ell(\gamma)|\gamma\inv(U\cap \varphi\inv(N))|=\ell(\gamma)|\gamma\inv(U)\cap (\varphi\circ\gamma)\inv(N)|>0,
\]
thus $\gamma\in \Gamma^+_{U\cap \varphi\inv(N)}$; but $U\cap \varphi\inv(N)=(\varphi|_U)\inv(N\cap \varphi(U))=u(N\cap \varphi(U))$ is $\mu$-null\footnote{Note that $\varphi(U)\subset K_0$, $u|_{K_0}$ is $L$-Lipschitz, and that $\mu|_U$ and $\mathcal H^2|_U$ are mutually absolutely continuous.}, implying $\operatorname{Mod}_2(\Gamma_1;\mu)=0$.

For any (absolutely continuous) curve $\gamma$ in $X$ the identities $|\gamma_t'|=|(\gamma|_{\gamma\inv(U)})'_t|$ and $(\varphi\circ\gamma)_t'=(\varphi\circ\gamma|_{\gamma\inv(U)})_t'$ hold for a.e.\ $t\in \gamma\inv(U)$. Moreover $\gamma|_{\gamma\inv(U)}=u\circ\varphi\circ\gamma|_{\gamma\inv(U)}$, since $\varphi|_U=u\inv|_U$. Thus, for $\gamma\notin\Gamma_1$ (i.e.\ for $\operatorname{Mod}_2$-a.e.\ $\gamma$ in $(X,\mu)$), we have that 
\begin{align*}
|\gamma_t'|=|(u\circ\varphi\circ\gamma|_{\gamma\inv(U)})_t'|=\md_{\varphi(\gamma_t)}(u|_{K_0})((\varphi\circ\gamma|_{\gamma\inv(U)})_t')=\md_{\varphi(\gamma_t)}u((\varphi\circ\gamma)_t')
\end{align*}
for a.e.\ $t\in \gamma\inv(U)$.

Given $v\in S^1$, set $\Gamma_v=\{ \gamma\subset B(z,r_0):\ \gamma'=v\}$, and note that
\begin{align*}
0<\operatorname{Mod}_2(\Gamma_v)\le \operatorname{Mod}_2(u(\Gamma_v);\mu)
\end{align*}
by Lemma \ref{lem:mod-est}. For $\operatorname{Mod}_2$-a.e.\ $\gamma\in u(\Gamma_v)$ we have $(\varphi\circ\gamma)'_t=v$ for a.e.\ $t\in \gamma\inv(U)$, and consequently for any $\xi\in (\R^2)^*$, 
\begin{align*}
\Phi(\xi,\gamma_t)\ge \frac{\xi((\varphi\circ\gamma)_t')}{|\gamma_t'|}=\frac{\xi(v)}{s_{\varphi(\gamma_t)}(v)}\quad \textrm{for a.e.\ }t\in \gamma\inv(U).
\end{align*}
Taking supremum over a countable dense set $D\subset S^1$ and arguing as in the proof of \cite[Lemma 4.2(2)]{EB-sou24} we obtain that
\begin{align}\label{eq:dual-norm-lower-bound}
\Phi(\xi,x)\ge s_{\varphi(x)}^*(\xi):=\sup\left\{\frac{\xi(v)}{s_{\varphi(x)}(v)}:\ v\in S^1\right\}=s^*_{\varphi(x)}(\xi), \quad \xi\in (\R^2)^*,
\end{align}
for $\mu$-a.e.\ $x\in U$. In particular, since $s_z\simeq_L|\cdot|_{Eucl}$, this shows that $(U,\varphi)$ is 2-independent. On the other hand 
\begin{align}\label{eq:dual-norm-upper-bound}
\frac{|\xi((\varphi\circ\gamma)_t')|}{|\gamma_t'|}=\xi\left(\frac{(\varphi\circ\gamma)_t'}{s_{\varphi(\gamma_t)}((\varphi\circ\gamma)_t')}\right)\le s_{\varphi(\gamma_t)}^*(\xi)\quad\textrm{for a.e.\ }t\in \gamma\inv(U)
\end{align}
for $\operatorname{Mod}_2$-a.e.\ $\gamma$. Together \eqref{eq:dual-norm-lower-bound} and \eqref{eq:dual-norm-upper-bound} imply
\begin{align*}
s_{\varphi(x)}(v)=(\Phi^x)^*(v),\quad v\in\R^2
\end{align*}
for $\mu$-a.e.\ $x\in U$. It follows that $s_z=(\Phi^{u(z)})^*$ for a.e.\ $z\in \varphi(U)$. This is a contradiction since $\Phi^{u(z)}$ (and therefore its dual norm) is an inner product norm for a.e.\ $z\in \varphi(U)$ by Lemma \ref{lem:IH-inner prod}. 
\end{proof}
%
%
%
%

%
%
%
%
\end{document}